\documentclass[a4paper,11pt]{amsart}

\usepackage{amsmath,amsfonts,amssymb,amsthm}
\usepackage{mathrsfs} 
\usepackage{extarrows} 

\usepackage[x11names,rgb,table]{xcolor}
\usepackage[naturalnames=true]{hyperref}
\hypersetup{
  colorlinks = true,          
  urlcolor = DeepSkyBlue4, linkcolor = Chartreuse4, citecolor = DarkOrange2
}

\usepackage{tikz}
\usepackage{tikz-cd}
\usetikzlibrary{calc,decorations.pathmorphing}
\usepackage{graphicx}
\usepackage{multicol}
\usepackage{faktor}
\usepackage{csquotes}
\usepackage{xspace}
\usepackage{bm}

\usepackage{listings}
\lstdefinelanguage{pmshell}
{
  basicstyle=\footnotesize\ttfamily,
  numbers=none,
  captionpos=b,
  showspaces=false,
  showstringspaces=false,
  escapechar=&,
  backgroundcolor=\color{black!10},
  framexleftmargin = 0.5em,
  framexrightmargin = 0.5em,
  framextopmargin = 0.5em,
  framexbottommargin = 0.5em,
  frame=single,
  framerule=0pt,
}
\lstdefinelanguage{pmoutput}
{
  basicstyle=\footnotesize\ttfamily,
  numbers=none,
  captionpos=b,
  showspaces=false,
  showstringspaces=false,
  escapechar=&,
}

\usepackage[colorinlistoftodos, bordercolor=orange, backgroundcolor=orange!20, linecolor=orange, textsize=scriptsize
]{todonotes}

\usetikzlibrary{circuits.logic.US,circuits.logic.IEC,fit}
\newcommand\addvmargin[1]{
  \node[fit=(current bounding box),inner ysep=#1,inner xsep=0]{};
}

\newtheorem{theorem}{Theorem}[section]

\theoremstyle{definition}

\newtheorem{example}[theorem]{Example}

\newtheorem{question}[theorem]{Question}
\newtheorem{remark}[theorem]{Remark}

\newcommand{\MM}{\mathbb{M}}

\newcommand{\RR}{\mathbb{R}}

\newcommand{\ZZ}{\mathbb{Z}}

\newcommand\polymake{\texttt{polymake}\xspace}

\newcommand\1{\textbf{1}}

\newcommand\SetOf[2]{\left\{#1 \mid #2\right\}}
\newcommand\smallSetOf[2]{\{#1 \mid #2\}}

\DeclareMathOperator{\Aut}{Aut}

\DeclareMathOperator{\Sym}{Sym}
\DeclareMathOperator{\conv}{conv}
\DeclareMathOperator{\pl}{pl}

\DeclareMathSymbol{\dash}{\mathord}{operators}{`\-}
\DeclareMathOperator{\GL}{GL}
\DeclareMathOperator{\link}{lk}

\newcommand\doi[1]{\href{http://dx.doi.org/#1}{\texttt{doi:#1}}}

\title{Stacky fans and tropical moduli in \polymake}
\author{Dominic Bunnett \and Michael Joswig \and Julian Pfeifle}

\address[Bunnett]{
  TU Berlin, Chair of Discrete Mathematics/Geometry
}
\email{bunnett@math.tu-berlin.de}
\address[Joswig]{
  TU Berlin, Chair of Discrete Mathematics/Geometry, and Max-Planck Institute for Mathematics in the Sciences, Leipzig
}
\email{joswig@math.tu-berlin.de}
\address[Pfeifle]{
  Departament de Matem\`atiques, Universitat Polit\`ecnica de Catalunya
}
\email{julian.pfeifle@upc.edu}

\thanks{%
  D.~Bunnett and M.~Joswig have been supported by Deutsche Forschungsgemeinschaft (EXC 2046: \enquote{MATH$^+$}, SFB-TRR 195: \enquote{Symbolic Tools in Mathematics and their Application}, and GRK 2434: \enquote{Facets of Complexity}).
  J.~Pfeifle has been supported by the grant PID2019-106188GB-I00 from the
Spanish Ministry of Education (MEC)
}

\begin{document}

\begin{abstract}
  We investigate geometric embeddings among several classes of stacky fans and algorithms, e.g., to compute their homology.
  Interesting cases arise from moduli spaces of tropical curves.
  Specifically, we show that the tropical honeycomb curves form a contractible sub-locus in the moduli of all tropical $K_4$-curves.
\end{abstract}

\maketitle

\section{Introduction}
Moduli spaces of algebraic curves and their compactifications form a research direction in algebraic geometry of ongoing interest.
In recent years many researchers contributed to understanding these classic moduli spaces via moduli spaces of tropical curves \cite{AbramoichCaporasoPayne:2015},\cite{ChanGalatiusPayne:2021}.
The latter are attractive since they admit descriptions in terms of polyhedral geometry and are thus algorithm friendly.
Here \emph{stacky fans} form a key concept; these are pieced together from topological quotients of polyhedral cones by linear actions of finite groups.
The purpose of this note is to report on an implementation of stacky fans in \polymake \cite{DMV:polymake}.
As an application we study the embedding of tropical plane curves of genus $3$ into the moduli space of all tropical genus $3$ curves.

We are indebted to Daniel Corey for his comments on a draft of this note.

\section{Stacky fans}
Consider $X_1 \subset \RR^{m_1}, \dots , X_k \subset \RR^{m_k}$ full-dimensional rational open polyhedral cones each with an action of a finite subgroup $G_i \subseteq \GL_{m_i}(\ZZ)$.
A \emph{stacky fan} is a topological space $X$ and a collection of maps
\[\alpha_i : \overline{X_i} / G_i \longrightarrow X\]
satisfying several conditions not detailed here.
Most importantly, the images of the $\alpha_i$ cover $X$ and are homeomorphisms when restricted to the interior~$X_i$.
For the full and lengthy definition we refer the reader to \cite[Definition 3.2]{Chan:2012}.
Roughly, one can think of a stacky fan as a fan where the constituent cones are replaced by finite quotients of cones.

To realise a stacky fan in \texttt{polymake}, one can expand the definition of a cone, that is, we allow a cone to possess as a property a group with a group action.
Given this perspective, the implementation is simple.

\begin{figure}[tb]
  \begin{minipage}{0.40\textwidth}
    \scalebox{1}{
\begin{tabular}{|c|c|}
  \hline
  Link of cell      &   Group
  \\ \hline$X_1=$
 \scalebox{0.3}{
\begin{tikzpicture}[baseline={(x.base)}]

\draw [line width =0.2pt] (-2,0) -- (2,0);
\draw [line width =0.2pt] (0,3) -- (2,0);
\draw [line width =0.2pt] (-2,0) -- (0,3);

\node (x) at (0,1) {};
\node [circle, fill, inner sep=3pt] at (-2,0) {};
\node [circle, fill, inner sep=3pt] at (2,0) {};
\node [circle, fill, inner sep=3pt] at (0,3) {};
\addvmargin{3mm}
  \end{tikzpicture} }  &     $S_3$
  \\ \hline $X_2=$
   \scalebox{0.3}{
\begin{tikzpicture}[baseline={(x.base)}]

\draw [line width =0.2pt] (-2,0) -- (2,0);
\draw [line width =0.2pt] (0,3) -- (2,0);
\draw [line width =0.2pt] (-2,0) -- (0,3);
\node (x) at (0,1) {};
\node [circle, fill, inner sep=3pt] at (-2,0) {};
\node [circle, fill, inner sep=3pt] at (2,0) {};
\node [circle, fill, inner sep=3pt] at (0,3) {};
\addvmargin{3mm}
  \end{tikzpicture} }  &   $ \langle(u,v) \rangle$
  \\ \hline $X_3=$
  \scalebox{0.3}{
 \begin{tikzpicture}[baseline={(x.base)}]

\draw [line width =0.2pt] (-2,0) -- (2,0);
\node (x) at (0,-0.3) {};
\node [circle, fill, inner sep=3pt] at (-2,0) {};
\node [circle, fill, inner sep=3pt] at (2,0) {};
\addvmargin{3mm}
\end{tikzpicture} }  &    $S_2$
  \\ \hline$X_4=$
    \scalebox{0.3}{
 \begin{tikzpicture}[baseline={(x.base)}]

\draw [line width =0.2pt] (-2,0) -- (2,0);
\node (x) at (0,-0.3) {};
\node [circle, fill, inner sep=3pt] at (-2,0) {};
\node [circle, fill, inner sep=3pt] at (2,0) {};
\addvmargin{3mm}
\end{tikzpicture} }  &    $\1$
  \\ \hline$X_5=$
    \scalebox{0.3}{
 \begin{tikzpicture}[baseline={(x.base)}]
 \node (x) at (0,-0.3) {};
\node [circle, fill, inner sep=3pt] at (0,0) {};
\addvmargin{3mm}
\end{tikzpicture} }  &     $\1$
 \\ \hline$X_6=$
    \scalebox{0.3}{
 \begin{tikzpicture}[baseline={(x.base)}]
 \node (x) at (0,-0.3) {};
\node [circle, fill, inner sep=3pt] at (0,0) {};
\addvmargin{3mm}
\end{tikzpicture} }  &     $\1$
  \\ \hline
  \end{tabular}
}
  \end{minipage}
  \hfill
  \begin{minipage}{0.55\textwidth}
    \scalebox{1}{
\begin{tikzpicture}

\draw [line width =0.3pt, fill=black!20] (-3,0.5) -- (-1,0.5) -- (-2,2) -- (-3,0.5);
\draw [line width =0.3pt, fill=black!20] (-3,-0.5) -- (-1,-0.5) -- (-2,-2) -- (-3,-0.5);

\draw [line width =0.3pt, fill=black!20] (1,1) -- (3,1) -- (3,2) -- (1,1);
\draw [line width =0.3pt, fill=black!20] (1,1) -- (3,-2) -- (3,1);

\draw [dashed] (-2,2) -- (-2,0.5);
\draw [dashed] (-2,-2) -- (-2,-0.5);
\draw [dashed] (-1,0.5) -- (-2.5,1.25);
\draw [dashed] (-3,0.5) -- (-1.5,1.25);

\draw [line width =1.2] (1,1) -- (3,1);
\draw [line width =1.2] (1,1) -- (3,-2);

\node [circle, fill, inner sep=1pt] at (-3,0.5) {};
\node [left] at (-3,0.5) {$x$};
\node [circle, fill, inner sep=1pt] at (-1,0.5) {};
\node [right] at (-1,0.5) {$y$};
\node [circle, fill, inner sep=1pt] at (-2,2) {};
\node [above] at (-2,2) {$z$};
\node [circle, fill, inner sep=1pt] at (-3,-0.5) {};
\node [left] at (-3,-0.5) {$u$};
\node [circle, fill, inner sep=1pt] at (-1,-0.5) {};
\node [right] at (-1,-0.5) {$v$};
\node [circle, fill, inner sep=1pt] at (-2,-2) {};
\node [below] at (-2,-2) {$w$};

\node [circle, fill, inner sep=1.5pt] at (1,1) {};
\node [circle, fill, inner sep=1.5pt] at (3,-2) {};

\node at (2.5,1.3) {$C_1$};
\node at (2.5,0) {$C_2$};
\node at (1.8,-0.8) {$C_4$};
\node at (1.1,1.7) {$C_3$};
\node [left] at (1,1) {$C_5$};
\node [below] at (3,-2) {$C_6$};

\draw[->] (1.2,1.8)  to [out=40,in=90] (2,1);

\draw [->] (-0.5,0) -- (0.5,0);

\end{tikzpicture}
}
  \end{minipage}
  \caption{Cells and the link of the simplicial stacky fan $X$ from Example~\ref{exmp:delta-2}.}
  \label{fig:delta-2}
\end{figure}
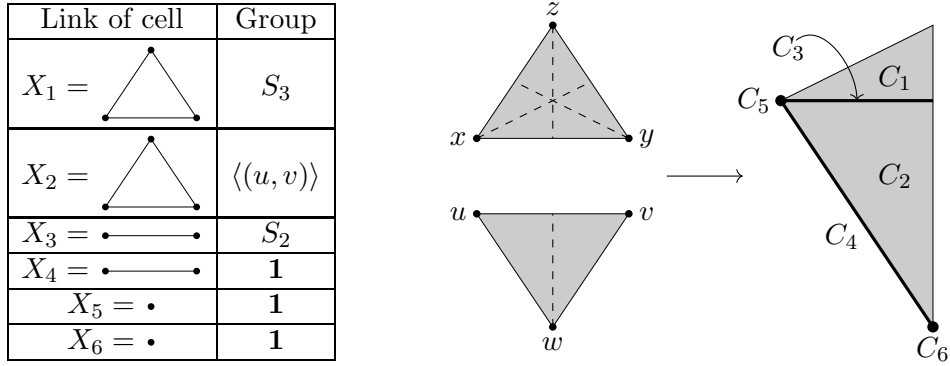

\begin{example}\label{exmp:delta-2}
  Consider the simplicial cones $X_1= \RR_{> 0}^{(x,y,z)}$ and $X_2 = \RR_{> 0}^{(u,v,w)}$ with actions of $G_1=\Sym\{x,y,z\}\cong\Sym(3)$ and $G_2=\langle (u\ v)\rangle \cong \Sym(2)$.
  Both quotient spaces
  \[C_1 := \faktor{X_1}{G_1} \quad \text{ and } \quad C_2 := \faktor{X_2}{G_2}\]
  can be identified with their respective fundamental domains.
  We then glue the face of $X_2$ which was folded in half by $G_2$ to any face of $X_1$; note that $\Sym(3)$ acts transitively on the set of facets of $X_1$.
  Denote the resulting space by $X$.
  This is well defined since the restriction of the actions to these faces agree.
  See Figure~\ref{fig:delta-2} for a complete description of the cells. 
  In the code below, we see how to express $C_2$ as a stacky fan in \polymake.
\begin{lstlisting}[language=pmshell]
($u,$v,$w) = (0,1,2);
$G2 = new group::PermutationAction(GENERATORS=>[[$v,$u,$w]]);
$X2 = new Cone(RAYS=>[[1,0,0],[0,1,0],[0,0,1]], GROUP=>
          new group::Group(HOMOGENEOUS_COORDINATE_ACTION=>$G2));
$C2 = stacky_fan($X2);
\end{lstlisting}
\end{example}

This implementation is apt given stacky fans' namesake, stacks, in particular quotient stacks, who remember the group action.
Stacky fans where each cone is simplicial and the groups act by permuting the unit vectors of $\RR^n$ are referred to as \emph{simplicial stacky fans}.

\begin{remark}
  Since each cone contains the origin, a stacky fan is always contractible.
  Usually, we will be interested in the \emph{link} (of the origin), which is defined as
  \[ \link(X) := \bigcup_{i=1}^k \alpha_i(X_i \cap H_i), \]
  where $H_i=\smallSetOf{x\in\RR^{m_i}}{\sum x_\ell=1}$ is the affine hyperplane which intersects the orthant $\RR_{\geq 0}^{m_i}$ in the standard simplex $\Delta^{(m_i-1)}$.
  The link of a simplicial stacky fan is a \enquote{symmetric $\Delta$-complex} in the sense of \cite{ChanGalatiusPayne:2021}.
  The topology of the link of a stacky fan is in general non-trivial and of great interest.
  Indeed, the homology of the link of the stacky fan moduli space of tropical curves describes certain pieces of the cohomology of moduli space of algebraic curves \cite{ChanGalatiusPayne:2021} to name just one application.
\end{remark}

Note that stacky fans differ from fans in two important ways:
there is no natural embedding into a vector space, and
the intersection of two faces can be a union of lower dimensional faces.
This is reminiscent of the distinction between simplicial complexes and $\Delta$-complexes.

We remark on one feature of the implementation.
\begin{remark}\label{remark:inherit-structure}
  Suppose that $P \subseteq \RR_{>0}^{n}$ is a polyhedral cone and let $G \subseteq \Sym(n)$ act on $\RR_{>0}^n$.
  One can consider the polyhedral complex generated by letting $G$ act on $P$ and all of the faces of $P$.
  Denote this complex $G \cdot P \subset \RR_{>0}^n$.
  Its image  in the quotient $\faktor{\RR_{>0}^n}{G}$ inherits the structure of a stacky fan.
  The data and structure of this stacky fan is contained inside the data of a cone with the property of a group action.
  Whilst it is not necessary for $P$ to be a fundamental domain for the action of $G$ on $G \cdot P$, there will be many computational advantages to this being the case.
  This construction of stacky fans appears in our motivating example of moduli of tropical plane curves.
\end{remark}

\subsection*{Barycentric subdivisions}
We now restrict our attention to simplicial stacky fans and their links.
A fundamental operation on stacky fans is the barycentric subdivision.
For instance, the second barycentric subdivision of a finite symmetric $\Delta$-complex yields a finite simplicial complex on which the group action is regular~\cite[Definition~II.1.2 and Exercise~III.5]{Bredon:1972}, which makes it a standard way to represent a topological space in the computer.

Any action of $G \subset \Sym(n)$ on $\RR^n_{>0}$ permutes the cones in the barycentric subdivision.
Moreover, no point of any one (relatively open) cell is in the orbit of any other point in that cell.
Thus the cells of the barycentric subdivision are mapped bijectively onto the image in the quotient.
One crucial remark is that a fundamental domain for the action is a union of barycentric cells.

The (maximal) cells of a barycentric subdivision of any finite cell complex corresponds to the set of (maximal) directed paths in the Hasse diagram of the complex.
Consequently, in \polymake barycentric subdivisions are implemented as a depth first search in a directed graph.

\section{Moduli of tropical curves}

A smooth tropical curve of genus $g$ is a trivalent metric graph on $2g - 2$ nodes with $3g - 3$ edges, which satisfies a certain \enquote{balancing condition}.
The moduli of such curves are the lengths of these $3g-3$ edges.
Let us recall the construction of the simplicial stacky fan which is the moduli space of tropical curves.

Fix such a trivalent graph $\Gamma$.
Then a parameter space for all possible edge lengths of $\Gamma$ is given by $\RR_{>0}^{|E(\Gamma)|} \cong \RR^{3g-3}_{>0}$.
To remove the redundancies in this parameter space we must consider the quotient by $\Aut(\Gamma)$ which we denote
\[C(\Gamma) = \faktor{\RR^{3g-3}_{>0}}{\Aut(\Gamma)} \enspace . \]
Gluing each of these cells together via contracting edge lengths we get the moduli space of tropical curves
\[\MM_g = \faktor{\coprod \overline{C(\Gamma)}}{\sim} \enspace .\]

For the full construction and detailed proof that this is, as claimed, a stacky fan, we refer to \cite{Chan:2012}.
The stacky fan constructed in Example \ref{exmp:delta-2} is in fact the moduli space of genus 2 tropical curves $\MM_2$.

A tropical plane curve is a tropical curve which can be realised in $\RR^2$.
These arise as dual graphs of regular subdivisions of lattice polygons.
We define the moduli space of tropical plane curves to be
\[\MM_g^{\pl} := \overline{\SetOf{C \in \MM_g}{ \exists C' \cong C \text{ planar }}} \enspace ,\]
where we take the topological closure in $\MM_g$.
The tropical plane curves of genus $g\leq 5$ have been classified by an extensive computation~\cite{BJMS15}; the arising list of graphs (without edge lengths) has been verified independently \cite[Corollary~1]{JoswigTewari:2002.02270}.

\begin{figure}[tb]\centering
  \scalebox{0.5}{
\begin{tikzpicture}
\coordinate (v0) at (2,0);
\coordinate (v1) at (6,0);
\coordinate (v2) at (4,1);
\coordinate (v3) at (4,3);
\draw (v0) -- (v1)node [midway, below] {\Large $v$};
\draw (v0) -- (v2)node [pos=0.6, above] {\Large $x$};
\draw (v0) -- (v3) node [midway, above left] {\Large $u$};
\draw (v1) -- (v2) node [pos=0.6, above] {\Large $z$};
\draw (v1) -- (v3) node [midway, above right] {\Large $w$};
\draw (v2) -- (v3) node [pos=0.4, left] {\Large $y$};
\node [circle, inner sep=2pt,fill] at (v0) {};
\node [circle, inner sep=2pt,fill] at (v1) {};
\node [circle, inner sep=2pt,fill] at (v2) {};
\node [circle, inner sep=2pt,fill] at (v3) {};
\node at (4,-0.75) {\Large $(000)$};

\coordinate (v4) at (7.5,0);
\coordinate (v5) at (10.5,0);
\coordinate (v6) at (7.5,3);
\coordinate (v7) at (10.5,3);
\draw (v4) -- (v5) node [midway,below] {\Large $x$};
\draw (v6) -- (v7) node [midway,above] {\Large $w$};
\draw (v4) to [out = 60, in= -60] (v6);
\draw (v4) to [out = 120, in= -120] (v6);
\draw (v5) to [out = 60, in= -60] (v7);
\draw (v5) to [out = 120, in= -120] (v7);
\node at (6.8,1.5) [] {\Large $u$};
\node at (8.2,1.5) [] {\Large $v$};
\node at (9.8,1.5) [] {\Large $y$};
\node at (11.2,1.7) [] {\Large $z$};
\node [circle, inner sep=2pt,fill] at (v4) {};
\node [circle, inner sep=2pt,fill] at (v5) {};
\node [circle, inner sep=2pt,fill] at (v6) {};
\node [circle, inner sep=2pt,fill] at (v7) {};
\node at (9,-0.75) {\Large $(020)$};

\coordinate (v8) at (12.5,0);
\coordinate (v9) at (12.5,3);
\coordinate (v10) at (14.5,1.5);
\coordinate (v11) at (15.5,1.5);
\draw (16.1,1.5) circle (0.6);
\draw (v10) -- (v11) node [midway,above] {\Large $y$};
\draw (v8) -- (v10) node [midway,below] {\Large $x$};
\draw (v9) -- (v10) node [midway,above] {\Large $w$};
\draw (v8) to [out = 60, in= -60] (v9);
\draw (v8) to [out = 120, in= -120] (v9);
\node at (11.8,1.3) {\Large $u$};
\node at (13.2,1.5) {\Large $v$};
\node at (16.1,2.1) [above] {\Large $z$};
\node [circle, inner sep=2pt,fill] at (v8) {};
\node [circle, inner sep=2pt,fill] at (v9) {};
\node [circle, inner sep=2pt,fill] at (v10) {};
\node [circle, inner sep=2pt,fill] at (v11) {};
\node at (14,-0.75) {\Large $(111)$};

\coordinate (v12) at (18.5,1);
\coordinate (v13) at (19.2,1);
\coordinate (v14) at (20.7,1);
\coordinate (v15) at (21.4,1);
\draw (18,1) circle (0.5);
\draw (21.9,1) circle (0.5);
\draw (v12)--(v13) node [midway, above] {\Large $v$};
\draw (v14)--(v15) node [midway, above] {\Large $y$};
\draw(v13) to [out=90,in=90](v14);
\draw(v13) to [out=-90,in=-90](v14);
\node at (19.875,0.25) [] {\Large $x$};
\node at (19.875,1.7) [] {\Large $w$};
\node at (17.25,1) {\Large $u$};
\node at (22.3,1.65) {\Large $z$};
\node [circle, inner sep=2pt,fill] at (v12) {};
\node [circle, inner sep=2pt,fill] at (v13) {};
\node [circle, inner sep=2pt,fill] at (v14) {};
\node [circle, inner sep=2pt,fill] at (v15) {};
\node at (19.875,-0.75) {\Large $(212)$};

\coordinate (v16) at (23.9,0.7);
\coordinate (v17) at (24.7,1.5);
\coordinate (v18) at (24.7,2.63137);
\coordinate (v19) at (25.5,0.7);
\draw (v16)--(v17) node [pos=0.35, above] {\Large $y$};
\draw (v18)--(v17) node [midway, left] {\Large $x$};
\draw (v19)--(v17) node [pos=0.3, above] {\Large $z$};
\draw (23.5,0.3) circle (0.55);
\draw (25.9,0.3) circle (0.55);
\draw (24.7, 3.18137) circle (0.55);
\node at (25.5,3.18137) {\Large $u$};
\node at (26.75,0.3) {\Large $w$};
\node at (22.65,0.3) {\Large $v$};
\node [circle, inner sep=2pt,fill] at (v16) {};
\node [circle, inner sep=2pt,fill] at (v17) {};
\node [circle, inner sep=2pt,fill] at (v18) {};
\node [circle, inner sep=2pt,fill] at (v19) {};
\node at (24.7,-0.75) {\Large $(303)$};
\end{tikzpicture}
}
\caption{The five trivalent graphs of genus 3, with labelled edges}
\label{fig:g3}
\end{figure}
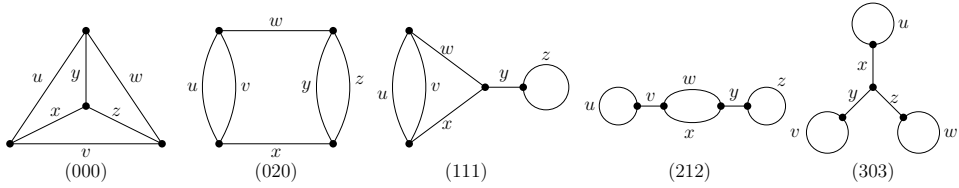

\subsection*{Genus $3$}

For $g=3$ there are five trivalent graphs (Figure~\ref{fig:g3}),
which correspond to five maximal dimensional cells in~$\MM_3$.
The space $\link(\MM_3)$ was shown to be homotopy equivalent to $S^5$ in \cite[Theorem 1.2]{ACP19}.
In fact, the authors prove that the union of all cells of $\link(\MM_3)$ other than the open cell~$\link(\MM_{K_4})$ containing the $K_4$-curves is a contractible subcomplex.
Subsequently it is shown that $\link(\MM_{K_4})$ is homotopy equivalent to a sphere.

Let us turn our attention to the moduli of tropical plane curves $\MM_3^{\pl}$.
Non-hyperelliptic algebraic genus 3 curves are plane quartics.
Analogously, all non-hyperelliptic smooth tropical plane curves of genus 3 are dual to triangulations of the triangle $T_4 = \conv\{(0,0),(4,0),(0,4)\}$.
We denote the corresponding moduli space by $\MM_{T_4}$.
These spaces are fairly intricate objects.
For instance, although $\dim \MM_{T_4} = \dim \MM_3 = 6$, the space $\MM_3^{\pl}$ is not pure dimensional due to realisable metrics on {\rm (111)}.
One interesting feature of this moduli space is that the graph {\rm (303)} does not appear in $\MM_3^{\pl}$; it cannot be realised in $\RR^2$, not even as a non-smooth tropical curve.
For the full computation we refer to \cite[Theorem 5.1]{BJMS15}.

Let us briefly discuss the remaining four cells of $\MM_3^{\pl}$ which do occur.
One can construct by hand a homeomorphism between the moduli cells $\MM_{212}^{\pl} \cong \MM_{212}$ of tropical curves of type {\rm (212)} and the planar locus.
Hence, we conclude that $\MM_{212}^{\pl}$ is contractible following \cite[Proposition 5.1]{ACP19}.
The next two cells, corresponding to {\rm (020)} and {\rm (111)}, have a more complicated structure, and this is beyond the scope of this article.
The subsequent section describes the embedding of the final cell $\MM^{\pl}_{000}$ into $\MM_3$.

\subsection*{Honeycombs and $K_4$-curves}

A $K_4$-curve is an algebraic curve whose tropicalisation has as a skeleton the complete graph on four vertices, which is labeled (000) in Figure~\ref{fig:g3}; see \cite{ChanPakawut:2017} for a detailed study of $K_4$-curves and their tropicalisations.
The locus of the tropicalisations of $K_4$-curves in $\MM_3$ occupies the maximal dimensional cell $\MM_{000}$, which we shall denote by $\MM_{K_4}$.
Let us describe $\MM_{K_4}$ explicitly.

Let $X_1 = \RR_{\geq 0}^6$ and label the coordinates $x,y,z,u,v,w$.
The group
\begin{equation}\label{eq:honeycomb-group}
  G = \langle (u\ v)(y\ z) ,\, (v\ w)(x\ y) ,\, (v\ x)(w\ y) \rangle
\end{equation}
is isomorphic to $\Sym(4)$; yet, as the three generators are double transpositions, $G$ is contained in the alternating group of degree six.
Then we have
\begin{equation}
  \MM_{K_4} = X_1 / G \enspace .
\end{equation}

The group $G$ permutes the six edges of the graph $K_4$.
In our setup, it is determined automatically as the automorphism group of~$K_4$,
which we specify in turn via its vertex-edge incidence matrix:
\begin{lstlisting}[language=pmshell]
($u, $v, $w, $x, $y, $z) = 0..5;
$skeleton = new IncidenceMatrix(
                [[$u,$v,$x],[$v,$w,$z],[$u,$w,$y],[$x,$y,$z]]);
$TG_K4 = new tropical::Curve(EDGES_THROUGH_VERTICES=>$skeleton);
\end{lstlisting}
With this we can find the f-vector of the corresponding stacky fan as follows:
\begin{lstlisting}[language=pmshell]
print tropical::stacky_fan($TG_K4)->STACKY_F_VECTOR;  
\end{lstlisting}
\begin{lstlisting}[language=pmoutput]
1 2 3 2 1
\end{lstlisting} 
The second barycentric subdivision of the link is a finite simplicial complex.
After appropriately identifying on the boundary of this complex, we obtain the moduli cell corresponding to $K_4$:
\begin{lstlisting}[language=pmshell]
$moduli_k4 = $TG_K4->MODULI_CELL;
print $moduli_k4->HOMOLOGY;
\end{lstlisting}
\begin{lstlisting}[language=pmoutput]
({} 0)
({} 0)
({} 0)
({} 0)
({} 0)
({} 1)
\end{lstlisting}
That output is the reduced integral homology of the $5$-sphere, which agrees with the aforementioned fact \cite{ACP19} that $\link(\MM_{K_4}) \simeq S^5$.
To get an idea of the size of the computation, we can access the face numbers of the moduli cell:
\begin{lstlisting}[language=pmshell]
print $moduli_k4->F_VECTOR;
\end{lstlisting}
\begin{lstlisting}[language=pmoutput]
356 6716 34320 71160 64800 21600
\end{lstlisting}

The \emph{honeycomb curves} form the locus in $\MM_3$ of tropically smooth $K_4$ plane curves.
The name is derived from the \enquote{honeycomb triangulations}; see the Section \enquote{Honeycombs} in \cite{BJMS15}.
The locus of honeycomb curves in $\MM_{K_4}$, computed in \cite[Theorem 5.1]{BJMS15}, is the polyhedral cone in $X_1$ defined by
\[\max\{x,y\} \leq u, \quad \max\{x,z\} \leq v \quad \text{ and }\quad \max\{y,z\} \leq w \enspace .\]
We denote this locus by $P\subset \RR_{\geq 0}^6$ and, abiding by the terminology in \cite{BJMS15}, call $P$ the \emph{honeycomb cone}.
As in Remark \ref{remark:inherit-structure}, $P$ defines a stacky fan, 
\[\MM_{K_4}^{\pl} \subset \MM_{K_4} \enspace .\]
This is the honeycomb locus embedded into the moduli space of tropical $K_4$-curves.
It turns out that the honeycomb locus has a fundamental domain which is convex and a union of barycentric cells; see Remark \ref{remark:inherit-structure}.
This will enable us, combined with~\cite{Bredon:1972} and the barycentric subdivision function, to compute the topology of $\MM_{K_4}^{\pl}$.
The incidence matrix is the same as above.
\begin{lstlisting}[language=pmshell]
($U,$V,$W,$X,$Y,$Z) =
  map { new Vector($_ ) } @{rows(unit_matrix(6))};
$D_ineqs = new Matrix([$U-$X, $U-$Y, $V-$X, $V-$Z, $W-$Y, $W-$Z]);
$TK4_planar = new tropical::Curve(
  EDGES_THROUGH_VERTICES=>$skeleton, INEQUALITIES=>$D_ineqs);
\end{lstlisting}
Now that we have the tropical curve as an object, we can compute its stacky f-vector, and the homology and f-vector of the moduli cell $\MM_{K_4}^{\pl}$:
\begin{lstlisting}[language=pmshell]
print stacky_fan($TK4_planar)->STACKY_F_VECTOR;
\end{lstlisting}%
\begin{lstlisting}[language=pmoutput]
4 8 10 7 2
\end{lstlisting}
\begin{lstlisting}[language=pmshell]
$mod_honey = $TK4_planar->MODULI_CELL;
print $mod_honey->F_VECTOR;
\end{lstlisting}
\begin{lstlisting}[language=pmoutput]
3809 72766 360654 723696 639360 207360
\end{lstlisting}
\begin{lstlisting}[language=pmshell]
print $mod_honey->HOMOLOGY;
\end{lstlisting}
\begin{lstlisting}[language=pmoutput]
({} 0)
({} 0)
({} 0)
({} 0)
({} 0)
({} 0)
\end{lstlisting}
The above computation is our first new result.
It says that the honeycomb locus is homologically trivial.
In fact, the next computation is even more interesting.
\begin{lstlisting}[language=pmshell]
print topaz::random_discrete_morse($mod_honey,rounds=>1);
\end{lstlisting}
\begin{lstlisting}[language=pmoutput]
{(<1 0 0 0 0 0> 1)}
\end{lstlisting}
This requires an explanation: Forman's discrete Morse theory provides fundamental methods to study the topology of a simplicial complex via its combinatorics and conversely \cite{Forman:1998}.
Bendetti and Lutz \cite{BenedettiLutz:2014} suggested a procedure to construct random discrete Morse functions, e.g., to certify that a complex is collapsible, a combinatorial property which is strictly stronger than contractibility.
Later this became part of a general heuristics for recognising spheres and balls \cite{JoswigLofanoLutzTsuruga:1405.3848}, which is implemented in \polymake.
The output above is a histogram of \enquote{discrete Morse vectors}.
Here we get one such vector, indicating a single cell in dimension zero, which is a point.
This provides a certificate that $\MM_{K_4}^{\pl}$ is collapsible.
For other input that random search may get stuck in a local minimum, whence it can be useful to increase the number of rounds.
Here this was unnecessary.
So that computation proves our main result:

\begin{theorem}\label{theorem:honeycomb-curves}
  The stacky fan $\MM_{K_4}^{\pl}$ of tropical honeycomb curves has face vector $(4,8,10,7,2)$, and the link $\link(\MM_{K_4}^{\pl})$ is contractible.
\end{theorem}

\subsection*{Computational complexity}
On a single thread on a 6816.61 BogoMips Intel Core i7-6700 CPU at 3.4 GHz, the computation time for the simplicial complex itself that triangulates the moduli cell is negligible;
the homology computation takes 4 seconds;
and the computation of the random discrete Morse functions takes around seven minutes.
To put this into perspective the following results are known:
By recent work of Tancer, deciding if a simplicial complex is collapsible is NP-complete \cite{Tancer:2016}.
Moreover, it is known to be undecidable to verify whether a complex is contractible.

\subsection*{Outlook}
In view of the known results for genus $g=1,2$ and our partial results for $g=3$, we ask the following.
Note also that $\MM_g^{\pl}$ is a subset of measure zero in $\MM_g$ for $g\geq 5$.
\begin{question}
  Is the link of the locus of tropical plane curves in $\MM_g$ always contractible, for arbitrary genus $g$?
\end{question}

It is tempting to address that question computationally.
Yet we found it challenging to extend the range of our methods beyond the cases presented here.
For instance, the next logical step would be to analyse the cells of type {\rm (020)} and {\rm (111)} in $\MM_3^{\pl}$.

\bibliographystyle{acm}
\bibliography{bib}

\end{document}